\documentclass[a4paper,11pt]{article}
\RequirePackage[T1]{fontenc}
\usepackage{palatino}
\usepackage{hyperref}
\hypersetup{colorlinks=true, linkcolor=blue,  anchorcolor=blue, citecolor=blue, filecolor=blue, menucolor=blue, urlcolor=blue}
\usepackage[english]{babel}
\usepackage[utf8]{inputenc}
\usepackage[font={small,sf},labelfont=bf]{caption}
\usepackage{amsmath,amssymb,amsfonts,mathrsfs,amsthm}
\usepackage{graphicx,pdfpages,enumitem,cite}
\theoremstyle{plain}
\newtheorem{theorem}{Theorem}[section] 
\newtheorem{lemma}[theorem]{Lemma} 
\newtheorem{proposition}[theorem]{Proposition} 
\theoremstyle{plain} 
\newtheorem{definition}{Definition}[section] 
\theoremstyle{definition} 
\newtheorem{remark}{Remark}[section] 
\usepackage{geometry}
\title{Stability of Forced-Damped Response in Mechanical Systems from a Melnikov Analysis} 
\author{Mattia Cenedese and George Haller
\thanks{Corresponding author: \href{mailto:georgehaller@ethz.ch}{georgehaller@ethz.ch}}\vspace{0.35cm}\\ 
Institute for Mechanical Systems, ETH Z\"urich,
\\ Leonhardstrasse 21, 8092 Z\"urich, Switzerland}
\date{\today}
\begin{document}
\maketitle
\begin{abstract}
	\noindent Frequency responses of multi-degree-of-freedom mechanical systems with weak forcing and damping can be studied as perturbations from their conservative limit. Specifically, recent results show how bifurcations near resonances can be predicted analytically from conservative families of periodic orbits (nonlinear normal modes). However, the stability of forced-damped motions is generally determined a posteriori via numerical simulations. In this paper, we present analytic results on the stability of periodic orbits that perturb from conservative nonlinear normal modes. In contrast with prior approaches to the same problem, our method can tackle strongly nonlinear oscillations, high-order resonances and arbitrary types of non-conservative forces affecting the system, as we show with specific examples.
\end{abstract}

\section{Introduction}
Conservative families of periodic orbits or nonlinear normal modes (NNMs) offer precious insights into multi-degree-of-freedom nonlinear vibrations \cite{Kerschen2009}. Indeed, experimental and numerical observations indicate that these periodic orbits shape the behavior of mechanical systems even after the addition of weak forcing and damping \cite{Touze2006,Peeters2011a,Peeters2011b,Ehrhardt2016,Renson2016b,Hill2017,Szalai2017}. A systematic mathematical analysis of this phenomenon via a Melnikov-type method has recently appeared in \cite{Cenedese2019}. A fundamental question of practical importance, however, is yet to be answered: what is the stability type of the forced-damped oscillations that perturb from their conservative counterparts. Different cases have been described in literature, ranging from the classic hysteretic amplitude-frequency curve of geometrically nonlinear structures \cite{Kovacic2008,Avramov2008} to the outbreak of unstable isolas and regions of the frequency response arising from nonlinear damping \cite{Ponsioen2019}. For detailed reviews, we refer the reader to the classic studies of Vakakis \cite{Vakakis1997,Vakakis2001,Vakakis2008}, Avramov and Mikhlin \cite{Avramov2011,Avramov2013}, Kerschen \cite{Kerschen2014} and to the introduction in \cite{Cenedese2019}.

In studies of mechanical systems with many degrees of freedom, stability is typically assessed numerically, either using Floquet's theory \cite{Floquet1883,GH1983} in the time domain, or by adopting Hill's method in the framework of harmonic balance \cite{Lazarus2010,Guillot2020}. Analytical perturbation expansions, such as the method of multiple scales \cite{NayfehM2007}, averaging \cite{Sanders2007}, the first-order normal form technique \cite{Touze2004} or the second-order normal form technique \cite{Neild2011}, are more suitable for low-dimensional oscillators and small-amplitude regimes. Spectral submanifolds \cite{Haller2016} provide another powerful tool for the analysis of forced-damped, nonlinear mechanical systems in a small enough neighborhood of the unforced equilibrium. Based on the relationship of these manifolds with their conservative limit \cite{Breunung2018,delaLlave2019,Veraszto2020}, known as Lyapunov subcenter manifold \cite{Lyapunov1992,Kelley1967}, the stability of forced-damped response can be analytically and numerically determined for small amplitudes \cite{Breunung2018,Ponsioen2019}.

The existence and stability of periodic orbits in near-integrable, low-dimensional systems has been studied via Melnikov-type methods, whose foundations date back to Poincaré \cite{Poincare1892}, Melnikov \cite{Melnikov1963} and Arnold \cite{Arnold1964}. In particular, for periodic orbits in single-degree-of-freedom systems, a Melnikov method can be used to detect perturbed, resonant periodic motions as well as their stability (see the excellent exposition in Chapter IV of \cite{GH1983} and the further analyses in \cite{Yagasaki1996,Bonnin2008}). Extensions to multi-degree-of-freedom systems have been available \cite{Veerman1985,Yagasaki1999}, but these approaches require low dimensionality and integrability before perturbation, neither of which holds for typical systems in structural dynamics.

In this work, we seek to fill this gap by complementing the analytical existence conditions for periodic responses in \cite{Cenedese2019} with assessment of their stability. Specifically, we develop a Melnikov method to establish the persistence of forced-damped periodic motions from conservative NNMs, without any restrictions on the number of degrees of freedom, the motion amplitude or the type of the non-conservative forces affecting the system. We assume that the conservative limit of the system is Hamiltonian, with no additional requirements on its integrability, and that the perturbation of this limit is small. In this setting, we derive analytical conditions to determine whether the oscillations created by small forcing and damping are unstable or asymptotically stable, thus generalizing the single-degree-of-freedom analysis of \cite{GH1983} to multi-degree-of-freedom systems. This extension is non-trivial and necessitates a new approach.

After stating and proving our mathematical results in a general setting for mechanical systems, we illustrate the validity of our predictions with numerical examples. These include subharmonic resonances in a gyroscopic system and isolated response generated by parametric forcing in a system of three coupled nonlinear oscillators.

\section{Setup}
\label{sec:S2}
We consider a mechanical system with $n$ degrees of freedom and denote its generalized coordinates by $q\in \mathbb{R}^n$, $n\geq 1$. We assume that this system is a small perturbation of a conservative limit described by the Lagrangian
\begin{equation}
\label{eq:Lagrcons}
\begin{array}{lcr}
L(q,\dot{q})=K(q,\dot{q})-V(q) &, &\displaystyle K(q,\dot{q})=\frac{1}{2}\langle\, \dot{q} \, , \, M(q)\dot{q} \rangle + \langle\, \dot{q} \, , \, G_1(q)\rangle + G_0(q),
\end{array}
\end{equation}
where, $M(q)$ is the positive definite, symmetric mass matrix, $K(q,\dot{q})$ is the kinetic energy and $V(q)$ the potential. The kinetic terms $G_1(q)$ and $G_0(q)$ may appear in conservative rotating mechanical systems after one factors out the cyclic coordinates whose corresponding angular momenta are conserved. The equations of motion for system (\ref{eq:Lagrcons}) take the form
\begin{equation}
\label{eq:MechSys}
D_t\left(\partial_{\dot{q}} L \right)-\partial_q L=\varepsilon Q(q,\dot{q},t;\delta,\varepsilon),
\end{equation}
where $D$ and $\partial$ denote total and partial differentiation with respect to the variable in subscript (for $D$, to all of them if the subscript is absent). Moreover, $\varepsilon\geq 0$ is the perturbation parameter and $\varepsilon Q(q,\dot{q},t;\delta,\varepsilon)=\varepsilon Q(q,\dot{q},t+\delta;\delta,\varepsilon)$ denotes a small, generic perturbation of time-period $\delta$.

As $L$ is a convex function of $\dot{q}$, the conservative system can be written in Hamiltonian form \cite{Arnold1989}. Introducing the generalized momenta
\begin{equation}
p=\partial_{\dot{q}} L=M(q)\dot{q}+G_1(q),
\end{equation}
we can express the velocities as $\dot{q}=F(q,p)=M^{-1}(q)(p-G_1(q))$ and the total energies as
\begin{equation}
\label{eq:Hdef}
H(q,p)=\langle\, p \, , \, F(q,p) \rangle -L(q,F(q,p))=\frac{1}{2}\langle\, p-G_1(q) \, , \, M^{-1}(q)(p-G_1(q)) \rangle - G_0(q) + V(q).
\end{equation}
Introducing the notation $x=(q,p)\in\mathbb{R}^{2n}$, we obtain the equations of motion in the form
\begin{equation}
\label{eq:NAutSysG}
\begin{array}{lr}
\dot{x}=JDH(x)+\varepsilon g(x,t;\delta,\varepsilon), & J = \begin{bmatrix} 0_{n\times n} & I_{n\times n} \\ -I_{n\times n} & 0_{n\times n} \end{bmatrix},
\end{array}
\end{equation}
where we assume that $H\in C^{r+1}$ with $r\geq 3$, while $g$ is $C^{r-1}$ in $t$ and $C^{r}$ in the other arguments. The vector fields in (\ref{eq:NAutSysG}) are defined as
\begin{equation}
\label{eq:FGDef}
\begin{array}{lr}
\displaystyle DH(x)= \begin{pmatrix} \partial_q H \\ \partial_p H \end{pmatrix}=\begin{pmatrix} -\partial_q L(q,F(q,p)) \\ F(q,p) \end{pmatrix}, &
\displaystyle g(x,t;\delta,\varepsilon)= \begin{pmatrix} 0 \\ Q(q,F(q,p),t;\delta,\varepsilon)  \end{pmatrix}.
\end{array}
\end{equation}
We assume that any further parameter dependence in our upcoming derivations is of class $C^r$ and that the model (\ref{eq:NAutSysG}) is valid in a subset $\mathcal{U}\subseteq\mathbb{R}^{2n}$ of the phase space . Trajectories of (\ref{eq:NAutSysG}) that start from $\xi\in\mathbb{R}^{2n}$ at $t=0$ will be denoted with $x(t;\xi,\delta,\varepsilon)=(q(t;\xi,\delta,\varepsilon),$ $p(t;\xi,\delta,\varepsilon))$. We will also use the shorthand notation $x_0(t;\xi)=(q_0(t;\xi),p_0(t;\xi))=x(t;\xi,\delta,0)$ for trajectories of the unperturbed (conservative) limit of system (\ref{eq:NAutSysG}). The equation of (first) variations for system (\ref{eq:NAutSysG}) about the solution $x(t;\xi,\delta,\varepsilon)$ reads
\begin{equation}
\label{eq:NAutSysG_FV}
\begin{array}{lr}
\dot{X}=\big( JD^2H(x(t;\xi,\delta,\varepsilon))+\varepsilon \partial_xg(x(t;\xi,\delta,\varepsilon),t;\delta,\varepsilon)\big)X, & X(0) = I_{2n\times 2n},
\end{array}
\end{equation}
whose solutions for $\varepsilon=0$ will be denoted as $X_0(t;\xi)=X(t;\xi,\delta,0)$. As long as $x_0(t;\xi) \in \mathcal{U}$, we recall that $H(x_0(t;\xi))=H(\xi)$ and that $X_0(t;\xi)$ is a symplectic matrix  \cite{deGosson2011}.

\section{The Melnikov method for resonant perturbation of normal families of conservative periodic orbits}
\label{sec:S3}
In this section, we briefly recall the results from Ref. 9 on the existence of perturbed periodic orbits for system (5) for small $\varepsilon>0$. We assume that there exists a periodic orbit $\mathcal{Z}\subset \mathcal{U}$ with minimal period $\tau>0$ solving system (\ref{eq:NAutSysG}) for $\varepsilon=0$. We can, therefore, consider this orbit as $m\tau$-periodic for a positive integer $m$ and we denote by $\Pi = X_0(m\tau;z)=X^m_0(\tau;z)$ the monodromy matrix of $\mathcal{Z}$ based at the point $z\in\mathcal{Z}$ evaluated along $m$ cycles of the periodic orbit $\mathcal{Z}$. The integer $m$ allows us to consider high-order resonances with the perturbation period. Note that $\Pi$ is nonsingular and, for a generic periodic orbit of an Hamiltonian system\cite{MO2017,Sepulchre1997}, $\Pi$ has at least two eigenvalues equal to $+1$. We further assume that $\mathcal{Z}$ is $m$-normal, i.e.,
\begin{definition}
\label{def:mnormality}
	A conservative periodic orbit $\mathcal{Z}$ is $m$-\textit{normal} if one of the following holds:
	\begin{enumerate}[nolistsep,label=\textit{(\roman*)}]
	\item The geometric multiplicity of the eigenvalue $+1$ of $\Pi$ is equal to $1$;
	\item The geometric multiplicity of the eigenvalue $+1$ of $\Pi$ is equal to $2$ and $JDH(z)\notin\mathrm{range}\left(\Pi-I\right)$ .	
\end{enumerate}
\end{definition}
Such orbits exist in one-parameter families \cite{Sepulchre1997,Doedel2003}. We give an illustration of $m$-normality in Fig. \ref{fig:S3_IM1}a where we indicate with $\mathcal{F}\subset \mathcal{U}$ the family emanating from $\mathcal{Z}$. Note that if a periodic orbit is not $m$-normal, then two or more families of periodic orbits may bifurcate from this periodic orbit. Moreover, a periodic orbit can be $1$-normal and not $m$-normal for some $m>1$, as in the case of subharmonic branching \cite{MO2017,Vanderbauwhede1997}. Any $m$-normal family of periodic orbits can be parametrized by the values of a scalar function $r$ that depends on a point of an orbit and its period. Common examples are $r(z,m\tau)=m\tau$ (period parametrization) or $r(z,m\tau)=H(z)$ (energy parametrization).

We aim to find periodic solutions of system (\ref{eq:NAutSysG}) for small $\varepsilon>0$ that may arise from $\mathcal{Z}$ or from the family $\mathcal{F}$.  Given a point $z\in\mathcal{Z}$ and a positive integer $l$ such that $m$ and $l$ are relatively prime, we constrain the initial conditions and periods of these perturbed solutions to satisfy the following resonance condition
\begin{figure}[t]
	\centering
	\includegraphics[width=1\textwidth]{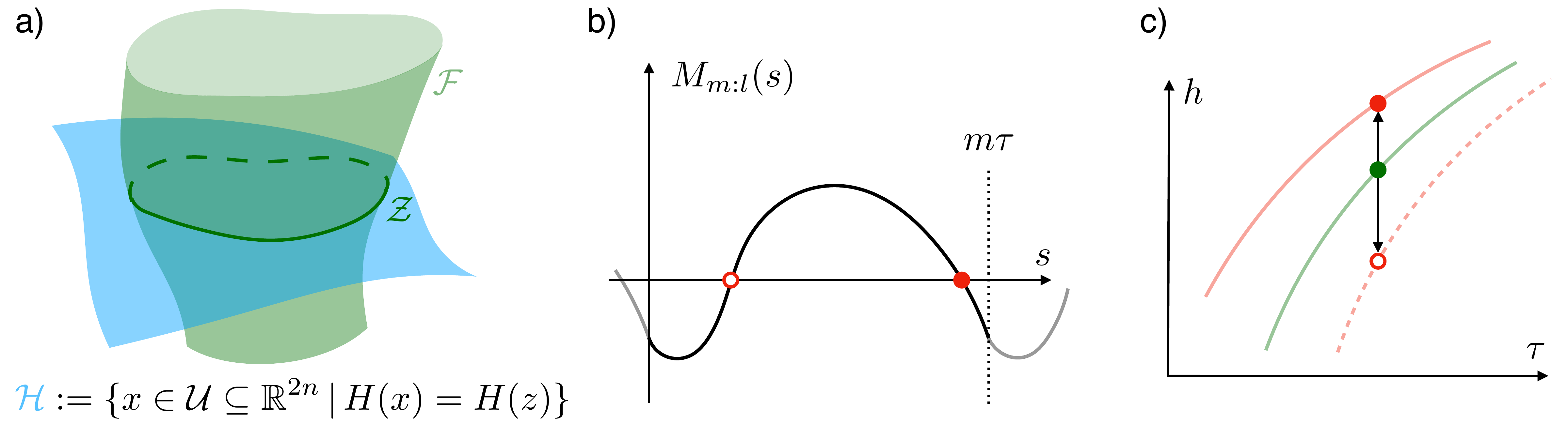}
	\caption{Review of the Melnikov analysis. Plot (a) shows a normal family $\mathcal{F}$ of periodic orbits (in green) and the orbit $\mathcal{Z}$ with its energy level $\mathcal{H}$. Plot (b) sketches one period of a Melnikov function (black line) with 2 simples zeros, corresponding to two perturbed orbits (red dots) from the conservative limit (green dot) in plot (c), displayed in the $(\tau,h)$ plane. Along with the backbone curve of $\mathcal{F}$ in green, the red shaded lines show frequency sweeps that can be detected via parameter continuation of the Melnikov function (see \cite{Cenedese2019} for more). }
	\label{fig:S3_IM1}
\end{figure}
\begin{equation}
\label{eq:rescond}
r(\xi,l\delta)=r(z,m\tau).
\end{equation}
The persistence problem of the periodic orbit $\mathcal{Z}$ is governed by the Melnikov function\cite{Cenedese2019}
\begin{equation}
\label{eq:MelFun}
\displaystyle M_{m:l}(s)=\int_0^{m\tau} \big\langle\, DH(x_0(t+s;z))\, , 
\, g(x_0(t+s;z),t;\tau m/l,0)\,\big\rangle\, dt .
\end{equation}
Specifically, if there exist a simple zero at $s_0$ of $M_{m:l}(s)$, i.e.,
\begin{equation}
\label{eq:simpzer}
\begin{array}{rcl}
M_{m:l}(s_0)=0 &,& DM_{m:l}(s_0)=M_{m:l}'(s_0)\neq 0,
\end{array}
\end{equation}
then $\mathcal{Z}$ smoothly persists with an initial condition $x_0(s_0;z)+O(\varepsilon)$ and a period $m\tau+O(\varepsilon)$. The period $\delta$ of the perturbation $g$ in system (\ref{eq:NAutSysG}) is then $O(\varepsilon)$-close to $\tau m/l$. Reference 9 also discusses the implications of different choices for $r$ in Eq. (8). As the Melnikov function (\ref{eq:MelFun}) is $m\tau$-periodic, it generically has another simple zero for $s\in(s_0,s_0+m\tau)$, as shown in Fig. \ref{fig:S3_IM1}b. Thus, at least two orbits bifurcate from the conservative limit (cf. Fig. \ref{fig:S3_IM1}c) when the conditions in Eq. (\ref{eq:simpzer}) are satisfied.

The Melnikov function (\ref{eq:MelFun}) can also be used to detect bifurcations under parameter continuation. For example, a quadratic zero generically signals a saddle-node bifurcation of periodic orbits or the formation of an isola for the frequency response for small $\varepsilon>0$. In \cite{Cenedese2019}, we also point out that our exact Melnikov analysis can be interpreted as an energy-balance principle proposed earlier\cite{Hill2015,Hill2016,Peter2018}. In particular, a perturbation expansion shows that the Melnikov function equals the work done by non-conservative forces along the periodic orbit $\mathcal{Z}$ of the conservative limit, i.e.,
\begin{equation}
\label{eq:MelFunEn}
\displaystyle M_{m:l}(s)=\int_0^{m\tau} \big\langle\, \dot{q}_0(t+s;z))\, , 
\, Q(q_0(t+s;z),\dot{q}_0(t+s;z),t;\tau m/l,0)\,\big\rangle\, dt ,
\end{equation}
where $\dot{q}_0(t+s;z)=\partial_pH(q_0(t+s;z),p_0(t+s;z))$.

In  Fig. \ref{fig:S3_IM1}c, we depict in red the two different branches of periodic orbits that originate from an $m$-normal periodic orbits family of the conservative limit (green). In the next sections, we shall derive analytical conditions under which the solid red branch contains asymptotically stable periodic orbits and the dashed branch is composed of unstable periodic orbits.

\section{Stability of perturbed periodic solutions}
\label{sec:S4}

In this section, we state our main mathematical results on the stability of solutions arising as perturbations from the conservative limit. We recall that the eigenvalues of the monodromy matrix, or Floquet multipliers \cite{Teschl2012,Chicone2000,GH1983}, determine the stability of a periodic orbit. In particular, the orbit is asymptotically stable if all of its Floquet multipliers lay inside the unit circle in the complex plane, while it is unstable if there exist a multiplier outside the unit circle. In accordance with the previous section, we now make the following assumption:
\begin{enumerate}[label=(A.\arabic*)]
\item The periodic orbit $\mathcal{Z}\subset \mathcal{U}$ of system (\ref{eq:NAutSysG}) for $\varepsilon=0$ is $m$-normal and, after a possible phase shift, its Melnikov function $M_{m:l}(s)$ defined in Eq. (\ref{eq:MelFun}) has a simple zero at $0$.
\end{enumerate}
By \cite{Cenedese2019}, assumption (A.1) implies the existence of an $l\delta$-periodic orbit, denoted $\mathcal{Z}_\varepsilon$, satisfying system (\ref{eq:NAutSysG}) for small $\varepsilon>0$, with $l\delta = m\tau+O(\varepsilon)$ and with initial condition $O(\varepsilon)$-close to $z$. Our first results is the following simple consequence of our assumptions.
\begin{proposition}
\label{prop:gen_instability}
If there exists an eigenvalue $\mu$ of $\Pi$ such that $|\mu|>1$, then $\mathcal{Z}_\varepsilon$ is unstable for $\varepsilon$ small enough.
\end{proposition}
\begin{proof}
The maximum distance between the Floquet multipliers of the orbit $\mathcal{Z}_\varepsilon$ and those of its conservative limit $\mathcal{Z}$ can be bounded with a suitable power of the parameter $\varepsilon$ (see, for example, Theorem 1.3 in Chapter IV of \cite{Stewart1990}). Hence, if there exists a multiplier $\mu$ of the conservative limit such that $|\mu|>1$, then, for $\varepsilon$ sufficiently small, $\mathcal{Z}_\varepsilon$ has a multiplier with the same property by the smooth dependence of the flow map on the parameter $\varepsilon$. 
\end{proof}
If the assumption of Proposition \ref{prop:gen_instability} is not satisfied, then perturbations of $\mathcal{Z}$ may result in orbits with different stability types. Henceforth we assume the following condition to be satisfied:
\begin{enumerate}[label=(A.\arabic*)]\setcounter{enumi}{1}
\item The eigenvalues of the monodromy matrix $\Pi$ lie on the unit circle in the complex plane.
\end{enumerate}
In the literature on Hamiltonian systems, an orbit satisfying (A.2) is called spectrally stable \cite{Sbano2009}. From a practical viewpoint, such orbits are of great interest as small perturbations of them may create asymptotically stable (and hence observable) periodic responses. We also need the next nondegeneracy assumption:
\begin{enumerate}[label=(A.\arabic*)]\setcounter{enumi}{2}
\item The algebraic and geometric multiplicities of the eigenvalue $+1$ of $\Pi$ are $2$ and $1$, respectively. 
\end{enumerate}
This guarantees that $\mathcal{F}$ can be locally parametrized either with the values of the first integral $H$ or with the period of the orbits \cite{Doedel2003}. In the former case, there exist a scalar mapping $T:\mathbb{R}\rightarrow\mathbb{R}^+$ that locally describes the minimal period of the orbits in $\mathcal{F}$ near $\mathcal{Z}$ as function of $H$. Moreover, $\tau = T(h)$ and $DT(h)=T'(h)\neq 0$ hold, where $h=H(z)$ is the energy level of $\mathcal{Z}$. 

To state further stability results, we need some definitions. Let $\mathcal{V}$ be a $2v$-dimensional invariant subspace for $\Pi$ and let $R_{\mathcal{V}}\in\mathbb{R}^{2n\times 2v}$ be a matrix whose columns form a basis of $\mathcal{V}$. We then have the following identity
\begin{equation}
\Pi R_{\mathcal{V}}=R_{\mathcal{V}}B_{\mathcal{V}}
\end{equation}
for a unique $B_{\mathcal{V}}\in\mathbb{R}^{2v\times 2v}$.
\begin{definition}
\label{def:stronginv}
We call $\mathcal{V}$ a strongly invariant subspace for $\Pi$ if $\mathrm{det}(B_\mathcal{V})=1$ and all the eigenvalues of $B_{\mathcal{V}}$ are not repeated in the spectrum of $\Pi$.
\end{definition}
Strongly invariant subspaces persist under small perturbations of $\Pi$, as shown in \cite{Stewart1990,Karow2014}, and we exploit this property in our technical proofs. The even dimensionality of any $\mathcal{V}$ in real form is a consequence of the fact that the eigenvalues of any symplectic matrix either appear in pairs $(\mu,1/\mu)$ or in quartets $(\mu,1/\mu,\bar{\mu},1/\bar{\mu})$ \cite{MH1992}. For the matrix $R_{\mathcal{V}}$, we call the left inverse 
\begin{equation}
\label{eq:symplli}
S_{\mathcal{V}} = \left( R_{\mathcal{V}}^{\top}J R_{\mathcal{V}} \right)^{-1}R_{\mathcal{V}}^{\top}J
\end{equation}
the symplectic left inverse of $R_\mathcal{V}$. As we later prove in the Appendix, $S_{\mathcal{V}}$ is well-defined.

Finally, we will use the notation
\begin{equation}
\label{eq:contrdef}
C_{\mathcal{V}} =-\frac{1}{m\tau v}\int_0^{m\tau}\mathrm{trace}\left(S_{\mathcal{V}} X^{-1}_0(t;z)\partial_x g(x_0(t;z),t;m\tau/l,0)X_0(t;z)R_{\mathcal{V}}\right)dt
\end{equation}
for the (local) volume contraction of the vector field $g$ along the orbit $\mathcal{Z}$ related to the strongly invariant subspace $\mathcal{V}$. 

The quantity $C_{\mathcal{V}}$, that serves as a nonlinear damping rate for the orbit $\mathcal{Z}_\varepsilon$, will turn out to have a key role in some of our upcoming conditions for the stability of $\mathcal{Z}_\varepsilon$. We also remark that $C_\mathcal{V}$ is invariant under changes of basis for $\mathcal{V}$. Indeed, we can define $\tilde{R}_\mathcal{V}=R_\mathcal{V}R_c$ for some invertible $R_c\in\mathbb{R}^{2v\times 2v}$ and obtain $\tilde{S}_\mathcal{V}=R_c^{-1}S_\mathcal{V}$, so that the invariance of the trace guarantees the one of $C_\mathcal{V}$. In particular, if $\mathcal{V}=\mathbb{R}^{2n}$, we have
\begin{equation}
\label{eq:divg}
C_{\mathbb{R}^{2n}} =-\frac{1}{m\tau n}\int_0^{m\tau}\mathrm{trace}\left( \partial_x g(x_0(t;z),t;m\tau/l,0) \right)dt.
\end{equation}

\subsection{Conditions for instability}
\label{sec:S4_1}

Due to assumption (A.3), the tangent space $\mathcal{T}$ of the family $\mathcal{F}$ at the point $z$ is the two-dimensional strongly invariant subspace for $\Pi$ related to its eigenvalues equal to $+1$. Due to the non-trivial Jordan block corresponding to these eigenvalues, the assessment of stability requires careful consideration. For a single-degree-of-freedom system ($n=1$), the tangent space is the only strongly invariant subspace for $\Pi$. For higher-dimensional systems ($n>1$), instabilities may develop also in the normal space $\mathcal{N}$ of the family $\mathcal{F}$ at the point $z$. The normal space is the $2(n-1)$-dimensional strongly invariant subspace for $\Pi$ related to its eigenvalues different from $+1$. The following theorem covers some generic cases of instability.
\begin{theorem}
\label{thm:instability} [\textbf{Sufficient conditions for instability}]. $\mathcal{Z}_{\varepsilon}$ is unstable for $\varepsilon >0$ small enough, if one of the following conditions is satisfied.
\begin{enumerate}[label=\textit{(\roman*)}]
\item Instabilities in $\mathcal{T}$: when $T'(h)M_{m:l}'(0)<0$ or when both $T'(h)M_{m:l}'(0)>0$ and $C_{\mathcal{T}}<0$.
\item Further instabilities ($n>1$): when $C_{\mathcal{V}}<0$ in a strongly invariant subspace $\mathcal{V}$ for $\Pi$.
\end{enumerate}
\end{theorem}
We prove this theorem in the Appendix. In statement \textit{(ii)}, $\mathcal{V}$ can be simply chosen as $\mathbb{R}^{2n}$ so its volume contraction can be directly computed with Eq. (\ref{eq:divg}). To identify instability in the normal space, one can then analyze any $\mathcal{V}\subseteq \mathcal{N}$. 

\begin{remark}
	\label{rmk:formulas0} 
	Theorem \ref{thm:instability} provides analytic expressions that allow to assess instability of the perturbed orbit, at least for generic cases, solely depending on the conservative limit, its first variation and the perturbative vector field. From a geometric viewpoint, we can thus detect instabilities whenever the volume of a strongly invariant subspaces shows expansion under action of the perturbed flow. 
\end{remark}

\begin{remark}
	\label{rmk:saddlenode} 
	The conservative limit $\mathcal{Z}$ exhibits a saddle-node bifurcation if either $T'(h)$ or $M_{m:l}'(0)$ is equal to zero, and if high order conditions are satisfied \cite{Cenedese2019,GolubitskySchaeffer1985}. In this case, a Floquet multiplier of the perturbed orbit crosses the unit circle along the positive real axis.
\end{remark}

\begin{remark}
	\label{rmk:noexistence} 
	If the existence of a perturbed periodic orbit arising from $\mathcal{Z}$ is unknown, one may still use Proposition \ref{prop:gen_instability} or statement \textit{(ii)} of Theorem \ref{thm:instability} to conclude the instability of any perturbed orbit originating from $\mathcal{Z}$.
\end{remark}

\subsection{Conditions for asymptotic stability}
\label{sec:S4_2}

While instability can be detected from a single condition, asymptotic stability from the linearization requires, by definition, more conditions. The next theorem provides such a set of conditions.
\begin{theorem}
	\label{thm:astability} [\textbf{Sufficient conditions for stability}]. Assume that, other than two eigenvalues equal to $+1$, the remaining eigenvalues of $\Pi$ are $n-1$ distinct complex conjugated pairs. For $k=1,...,\,n-1$, denote by $\mathcal{N}_k\subset \mathcal{N}$ the strongly invariant subspace for $\Pi$ related to the $k$th pair of eigenvalues. Then, $\mathcal{Z}_{\varepsilon}$ is asymptotically stable for $\varepsilon>0$ small enough if each of the following conditions holds
	\begin{enumerate}[label=\textit{(\roman*)}]
	\item $T'(h)M_{m:l}'(0)>0$ and $C_{\mathcal{T}}>0$,
	\item $C_{\mathcal{N}_k}>0$ for all $k=1,...,\,n-1$.
	\end{enumerate} 
\end{theorem}
We prove this theorem in the Appendix. The conditions \textit{(i)} in Theorems \ref{thm:instability} and \ref{thm:astability} apply for systems with one degree of freedom. These conditions are consistent with the ones derived in \cite{GH1983,Yagasaki1996}, up to sign changes due to the shift $s$ present in the Melnikov function. For $n=1$, the divergence of the perturbation, cf. Eq. (\ref{eq:divg}), is sufficient to determine the volume contraction since $C_{\mathbb{R}^{2}} \equiv C_{\mathcal{T}}$.

Simple zeros of the Melnikov function typically appear in pairs moving in opposite directions, as in Fig. \ref{fig:S3_IM1}b. Thus, assuming a pair of simple zeros and positive volume contractions in Theorem \ref{thm:astability}, two periodic orbits bifurcate from $\mathcal{Z}$ for $\varepsilon>0$ small enough. One of them is unstable and the other asymptotically stable, as argued in the end of section \ref{sec:S3} in Fig. \ref{fig:S3_IM1}c. This analytical conclusion matches with the results of several experimental and numerical studies present in the literature. Figure \ref{fig:S4_IM1} summarizes our results on stability showing the perturbation (in red) of the Floquet multipliers of the conservative limit (in green).

\begin{remark}
	\label{rmk:bifpert} 
	If $C_{\mathcal{N}_k}=0$, then $\mathcal{Z}_{\varepsilon}$ generically undergoes a Neimark-Sacker or torus bifurcation \cite{Arnold1988,Kuznetsov1995,MO2017}. High-order nondegeneracy conditions have to be satisfied, but this bifurcation implies that a resonant torus appears near $\mathcal{Z}_{\varepsilon}$. As they might be attractors for system (\ref{eq:NAutSysG}), their identification is relevant for the frequency response, as discussed in \cite{Amabili2007,Renson2015,Kuether2015}.
\end{remark}

\begin{remark}
	\label{rmk:othercases} 
	Theorem \ref{thm:astability} does not discuss in detail cases in which $\Pi$ admits either $-1$ as an eigenvalue or a repeated complex conjugated pairs. These configurations indicate that $\mathcal{Z}$ may be an orbit at which a period doubling or a Krein bifurcation, respectively, occurs \cite{MO2017}. One may still provide analytic expressions to evaluate the stability of $\mathcal{Z}_{\varepsilon}$, but we leave the discussion of these non-generic cases to dedicated examples.
\end{remark}
\begin{figure}[t]
	\centering
	\includegraphics[width=1\textwidth]{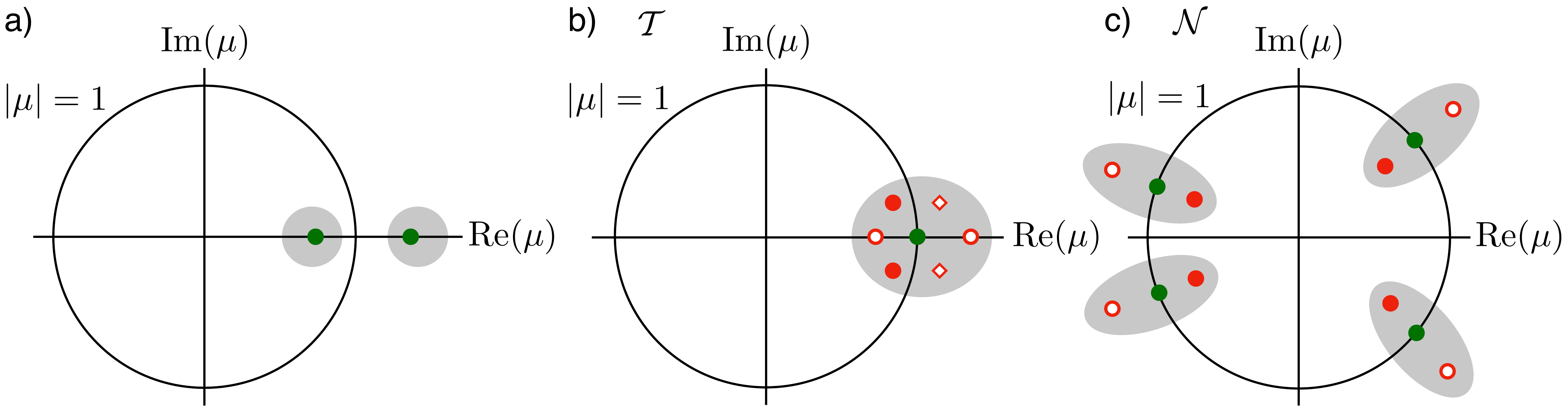}
	\caption{Summary of our stability results in terms of the perturbation (in red) of the Floquet multipliers of the conservative limit (in green). Grey shaded areas show possible perturbation zones. Plot (a) refers to Proposition \ref{prop:gen_instability} where the conservative limit has a real pair $(\mu,1/\mu)$ of Floquet multipliers. In contrast, plot (b) refers to the tangent space and illustrates the conditions \textit{(i)} of Theorems \ref{thm:instability} and of \ref{thm:astability} (here with red dots). In the former case, two different white-faced markers are used for the two possible instabilities. For the normal space, plot (c) shows the conditions \textit{(ii)} of Theorems \ref{thm:instability} (with red rings) and of \ref{thm:astability} (with red dots).}
	\label{fig:S4_IM1}
\end{figure}

\subsection{Determining contraction measures}
\label{sec:S4_3}

Especially for system with a large number of degrees of freedom, the volume contraction (or nonlinear damping rate) formula for $C_\mathcal{V}$ when $\mathcal{V}\subset\mathbb{R}^{2n}$ may be difficult to evaluate due to the presence of $X_0$, its inverse and the required subspace identification. Regarding the perturbative vector field as defined in Eq. (\ref{eq:NAutSysG_FV}), the following proposition illustrates the simple case of uniform volume contraction.
\begin{proposition}
\label{prop:uniformcontraction}
Let $\mathcal{V}$ be a strongly invariant subspace for $\Pi$. If $\partial_q Q(q,F(q,p),t;m\tau/l,0)$ is a symmetric matrix-valued function and $\partial_p Q(q,F(q,p),t;m\tau/l,0)=-\alpha I$, then $C_{\mathcal{V}} =\alpha$.
\end{proposition}
We prove this statement in the Appendix. In the Hamiltonian literature, the assumptions of Proposition \ref{prop:uniformcontraction} hold for conformally symplectic flows \cite{McLachlan2001,Calleja2013} under appropriate forcing.

For example, the condition of uniform contraction is satisfied in mechanical systems with $G_1\equiv 0$ when the leading-order perturbation terms are pure forcing and Rayleigh-type dissipation proportional to the mass matrix, i.e. $Q(q,\dot{q},t;\gamma,0) = f(t;\gamma) - \alpha M(q)\dot{q}$. However, such uniform volume contraction is an overly simplified damping model for practical applications and hence may only be relevant for numerical experiments.

\section{Examples}
\label{sec:S5}
\subsection{Subharmonic response in a gyroscopic system}
\label{sec:S5_1}
The equations of motion of the two-degree-of-freedom system in Fig. \ref{fig:S5_IM1} read
\begin{equation}
\begin{array}{c}
\displaystyle m_b\ddot{q} + 2G\dot{q} -m_b\Omega^2q+DV(q) = \hat{Q}(q,\dot{q},t) ,  \\ 
\displaystyle G=m_b\begin{bmatrix}0 & -\Omega \\ \Omega & 0 \end{bmatrix},  \,\,\,\,\, V(q)=\frac{1}{2} \sum_{j=1}^4 k_j\big(l_j(x,y)-l_{0}\big)^2,  \\ 
\displaystyle l_{1,3}(x,y) = \sqrt{(l_0\pm x)^2+y^2}  ,  \,\,\,\,\,   l_{2,4}(x,y) = \sqrt{x^2+(l_0\pm y)^2},
\end{array}
\end{equation}
where $q=(x,y)^\top$ are the generalized coordinates with respect to a reference frame rotating with constant angular velocity $\Omega$. We assume that the Lagrangian component $\hat{Q}$ contains all the small, non-conservative forces acting on the system as follows:
\begin{equation}
\hat{Q}(q,\dot{q},t) = \varepsilon \big(Q_{d,\alpha}(q,\dot{q}) + Q_{d,\beta}(q,\dot{q}) +Q_{f}(t) \big),
\end{equation}
\begin{itemize}
	\item uniform dissipation linearly depending on the absolute velocities of the mass $m_b$
	\begin{equation}
	\varepsilon Q_{d,\alpha}(q,\dot{q}) = -\varepsilon \alpha m_b (\dot{q}+m_b^{-1}G q );
	\end{equation}
	\item stiffness-proportional dissipation for the spring-damper elements, i.e. $c_j=\varepsilon \beta k_j$ for $j=1,... \,4$ and 
	\begin{equation}
	\begin{array}{c}
	\displaystyle \varepsilon Q_{d,\beta}(q,\dot{q}) = -\varepsilon \beta C(q)\dot{q},   \\ \\
	\displaystyle C(q)=\sum_{j=1}^4 k_j \begin{bmatrix} \big(\partial_x l_j(x,y)\big)^2 & \partial_x l_j(x,y)\partial_y l_j(x,y) \\ \partial_x l_j(x,y)\partial_y l_j(x,y)  & \big(\partial_y l_j(x,y)\big)^2 \end{bmatrix} ;  
	\end{array}
	\end{equation}
	\item mono-harmonic forcing of frequency $l\Omega$
	\begin{equation}
	\varepsilon Q_{f}(t) = \varepsilon e \begin{pmatrix} +\cos(l \Omega t) \\ - \sin(l \Omega t)  \end{pmatrix},\,\,\,\,\,\, l\in\mathbb{N}.
	\end{equation}
\end{itemize}
This simple model with strong geometric nonlinearities finds application in the fields of rotordynamics \cite{Ishida2012} or gyroscopic MEMS \cite{Younis2011}. Here, sinusoidal forces whose frequencies clock at multiple of the rotating angular frequencies either appear due to diverse effects \cite{Ehrich1988,Hamed2016} (e.g. asymmetries, nonrotating loads or multiphysical couplings) or are purposefully inserted in the system. From a physical standpoint \cite{Crandall1970}, the dissipation controlled by the coefficient $\alpha$ models, for example, radiation damping, while $\beta$ governs material or structural damping.

Introducing the transformation $p = m_b\dot{q} + Gq$, the Hamiltonian for the conservative limit takes the form
\begin{equation}
H(q,p) = \frac{1}{2m_b}\langle p,p \rangle - \frac{1}{m_b} \langle p,Gq \rangle+V(q).
\end{equation}
Under the assumption $m_b=1$, the equivalent, first-order equations for the two-degree-of-freedom system in Fig. \ref{fig:S5_IM1} are
\begin{equation}
\begin{array}{cc}
\begin{aligned}
\label{eq:mechsys51}
\displaystyle\dot{q}  &  =  - Gq + p , \\
\displaystyle \dot{p} &  =  - DV(q) - G p + \varepsilon \big( Q_f(t)-\alpha p -\beta C(q) (p - Gq) \big) .
\end{aligned}
\end{array}
\end{equation}
For our analysis, we further assume that $\Omega = 0.942$, $l_0=1$, $k_1 = 1$, $k_2 = 4.08$ , $k_3 = 1.37$ , $k_4 = 2.51$ and $e = 1$.
\begin{figure}[t]
	\centering
	\includegraphics[width=1\textwidth]{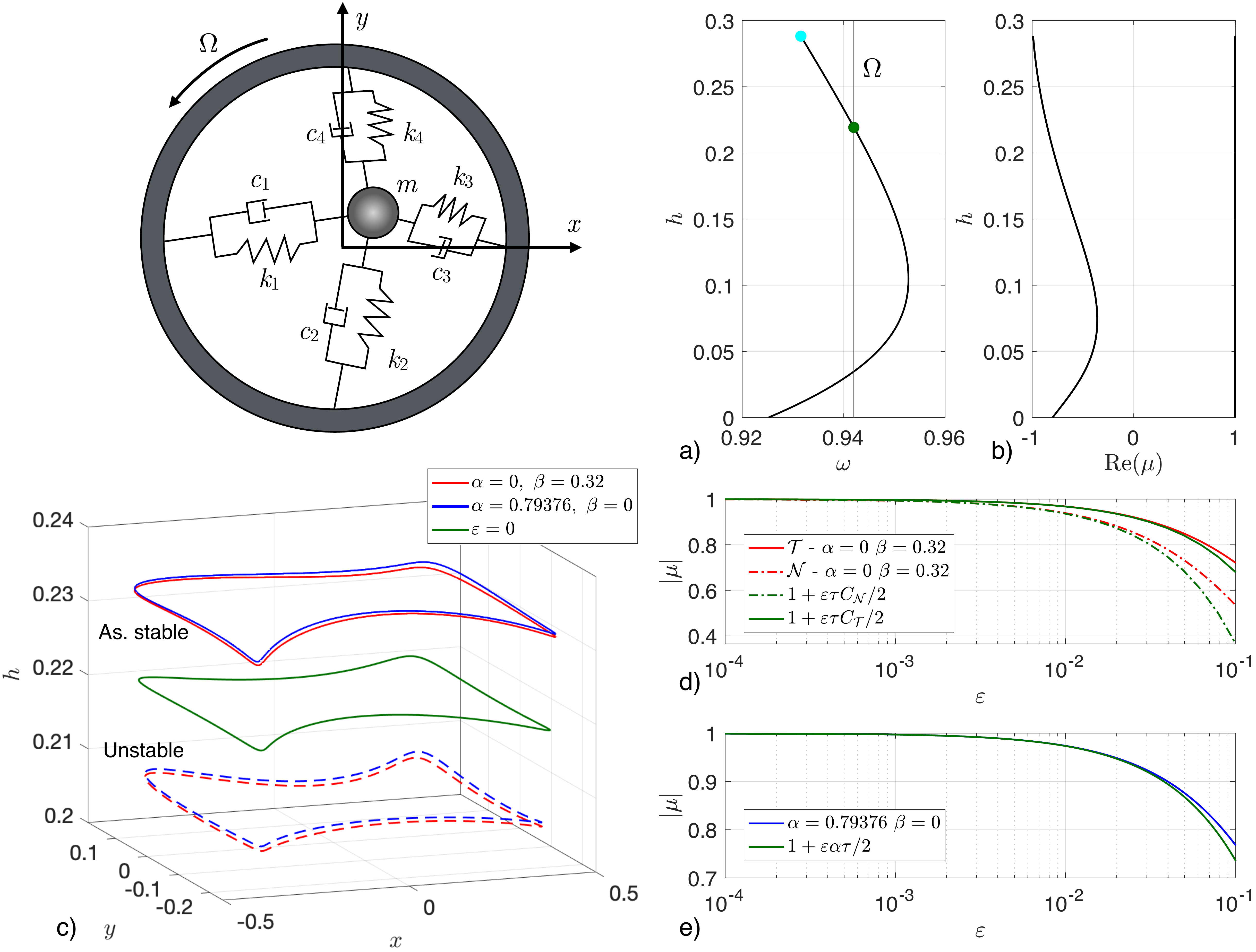}
	\caption{Plot (a) shows the conservative backbone curve in terms of frequency and value of the first integral for the family of periodic solutions for the dynamical system (\ref{eq:mechsys51}), while plot (b) the real part of the two couple of Floquet multiplier of the family. The green dot in plot (a) corresponds to the periodic orbit illustrated in plot (c) in terms of the coordinates $(x,y)$ and the value of the first integral. This plot also shows the stable (solid lines) and unstable (dashed lines) periodic orbits bifurcating from the conservative limit at $\varepsilon=0.01$. The blue lines indicate the orbit for $\alpha = 0.76376$ and $\beta = 0$, while red ones for $\alpha = 0$ and $\beta = 0.32$. With consistent colors, plots (d) and (e) show the evaluation of the absolute value of the Floquet multipliers whose analytic predictions are depicted in green. In plot (d), solid lines refers to the tangent space $\mathcal{T}$, while dashed-dotted ones to the normal space $\mathcal{N}$.}
	\label{fig:S5_IM1}
\end{figure}
We begin with the study of the conservative limit ($\varepsilon=0$), in which the origin is an equilibrium with the non-resonant linearized frequencies $(0.92513,\,3.1431)$. Hence, according to the Lyapunov subcenter theorem \cite{Lyapunov1992}, two families of periodic orbits (nonlinear normal modes) emanate from the origin; we focus our attention on the one related to the slowest linearized frequency. By performing numerical continuation, we obtain the periodic orbit family of Fig. \ref{fig:S5_IM1}a (backbone curve) described in terms of the oscillation frequency $\omega = 2\pi/\tau$ and the value of the first integral $h$. We also plot the real part of the Floquet mutipliers along the family in Fig. \ref{fig:S5_IM1}b. We stop continuation at the cyan point in Fig. \ref{fig:S5_IM1}a where a period doubling bifurcation occurs. 

The backbone curve in Fig. \ref{fig:S5_IM1}a crosses two times the vertical line marking the rotating angular velocity $\Omega$. We concentrate on the high-energy crossing point (depicted with a green dot), where the family shows a softening trend, i.e., $T'>0$ holds at this location. The corresponding periodic orbit is 1-normal and satisfies assumptions (A.2) and (A.3) of section \ref{sec:S3}.Moreover, this trajectory features a non-negligible third harmonic so that $1:3$ resonances with external forcing may occur, depending on the damping strength. Therefore, we fix $l=3$ and we study forced-damped periodic orbits that may survive from this high-energy crossing point for the two damping mechanisms we have in our model.

For $\alpha = 0.76376$ and $\beta = 0$, the Melnikov function (\ref{eq:MelFun}) evaluated for this periodic orbit reads
\begin{equation}
\label{eq:Mfunex1}
M_{1:3}(s) = 1.4402\cos(3\Omega s)- 1.1553.
\end{equation}
This function has six simple zeros, but, as shown in \cite{Cenedese2019}, they correspond to two perturbed orbits that occur as the amplitude of the work done by the forcing is greater than the dissipated energy, when evaluated at the conservative limit. Specifically, the zeros featuring a negative $M_{1:3}'(s)$ are related to an unstable orbit according to Theorem \ref{thm:instability}, while the others have a positive Melnikov-function derivative so that, due to Proposition \ref{prop:uniformcontraction} and to Theorem \ref{thm:astability}, they signal an asymptotically stable periodic orbit. Along with the conservative limit in green, we plot these perturbed orbits using blues lines in Fig. \ref{fig:S5_IM1}c (solid for the asymptotically stable and dashed for the unstable) that have been obtained by setting $\varepsilon=0.01$ in a direct numerical simulation with the periodic orbit toolbox of \textsc{coco} \cite{Dankowicz2013}. Qualitatively speaking, these periodic orbits are very similar to that of the conservative limit, but the average value of the first integral along them is higher (asymptotically stable orbit) or lower (unstable one).

By setting the damping values $\alpha = 0$ and $\beta = 0.32$, one retrieves the same Melnikov function as in Eq. (\ref{eq:Mfunex1}). In particular, the dissipated energy is equal to the case $\alpha = 0.76376$ and $\beta = 0$. Thus, again, a stable and an unstable periodic orbit bifurcates from the conservative limit since the volumte contractions are both positive. Again, direct numerical simulations with $\varepsilon=0.01$ verify our predictions: we plot the asymptotically stable and unstable periodic orbits in Fig. \ref{fig:S5_IM1}c with red solid and red dashed lines, respectively.

The nonlinear damping rates (or volume contractions) $C_\mathcal{V}$ can be used to estimate the Floquet multipliers of perturbed solutions. As shown in the proofs reported in the Appendix, the absolute value of a complex conjugated pair of eigenvalues arising from the perturbation of a strongly invariant subspace $\mathcal{V}$ reads
\begin{equation}
|\mu| = \sqrt{1 - \varepsilon \tau C_\mathcal{V}+ o(\varepsilon)}=1-\varepsilon \frac{\tau}{2}C_\mathcal{V} + o(\varepsilon).
\end{equation}
We illustrate these predictions for the asymptotically stable orbits of Fig. \ref{fig:S5_IM1}c. The green lines in Fig. \ref{fig:S5_IM1}d show these estimates for $\mathcal{T}$ (solid line) and $\mathcal{N}$ (dashed-dotted line) that are in good agreement with the multipliers computed within the perturbed system for the case $\alpha = 0$ and $\beta = 0.32$, plotted in red. Figure \ref{fig:S5_IM1}e represents the analogous curves for the case $\alpha = 0.76376$ and $\beta = 0$. Here, $C_\mathcal{T}=C_\mathcal{N}=\alpha$ and the green line depicts predictions from the conservative limit, while the blue one shows results from simulations in the forced-damped setting. We remark that, even though the dissipated energy is the same, stiffness-related damping provides higher volume contraction values with respect to the case of uniform damping. As expected, for high values of $\varepsilon$, our first-order computations may not be sufficient to adequately estimate the modulus of the perturbed multipliers.

For the low-energy crossing point between the backbone curve and the $\Omega$ vertical line, the Melnikov function is always negative. Therefore, no perturbed solution arises from this conservative limit for sufficiently small $\varepsilon$.

\subsection{Isolated response due to parametric forcing}
\label{sec:S5_2}
In this section, we study the parametrically-forced, three-degree-of-freedom system of nonlinear oscillators shown in Fig. \ref{fig:S5_IM2}. Parametric forcing \cite{NayfehM2007} finds notable applications in the field of MEMS \cite{Rhoads2005,Younis2011}. By assuming a linear damping proportional to the mass matrix and unit masses, the equations of motion in Hamiltonian form read
\begin{equation} 
\label{eq:systS52}
\begin{array}{cc}
\begin{aligned}
\dot{q}   & = p, \\
\dot{p}_1 & =  -k(q_1-q_2) - k/3 q_1 - a q_1^2 - b q_1^3 -\varepsilon \alpha p_1,\\ 
\dot{p}_2 & =  -k(q_2-q_1) -k(q_2-q_3) -\varepsilon \alpha p_2, \\
\dot{p}_3 & =  -k(q_3-q_2) +\varepsilon (q_3f(t;\Omega)-\alpha p_3), 
\end{aligned} &
\displaystyle f(t;\Omega) = \frac{4}{\pi}\sum_{j=1}^{3} \frac{1}{2j-1} \sin\big((2j-1)\Omega t \big),
\end{array}
\end{equation}
where $q,p\in\mathbb{R}^3$, $k=1$, $a =-1/2 $, $b=1$ and $\alpha,\varepsilon >0$. The nonlinear behavior in this example arises from the material nonlinearity of the left-most spring in Fig. \ref{fig:S5_IM2}. We expect the appearance of isolas in the frequency response, at least for small $\varepsilon$. Indeed, as the forcing amplitude is controlled by $q_3$, it is necessary to exceed a threshold on the motion amplitude for the work done by the forcing to overcome energy dissipation by the damping.
\begin{figure}[t]
	\centering
	\includegraphics[width=1\textwidth]{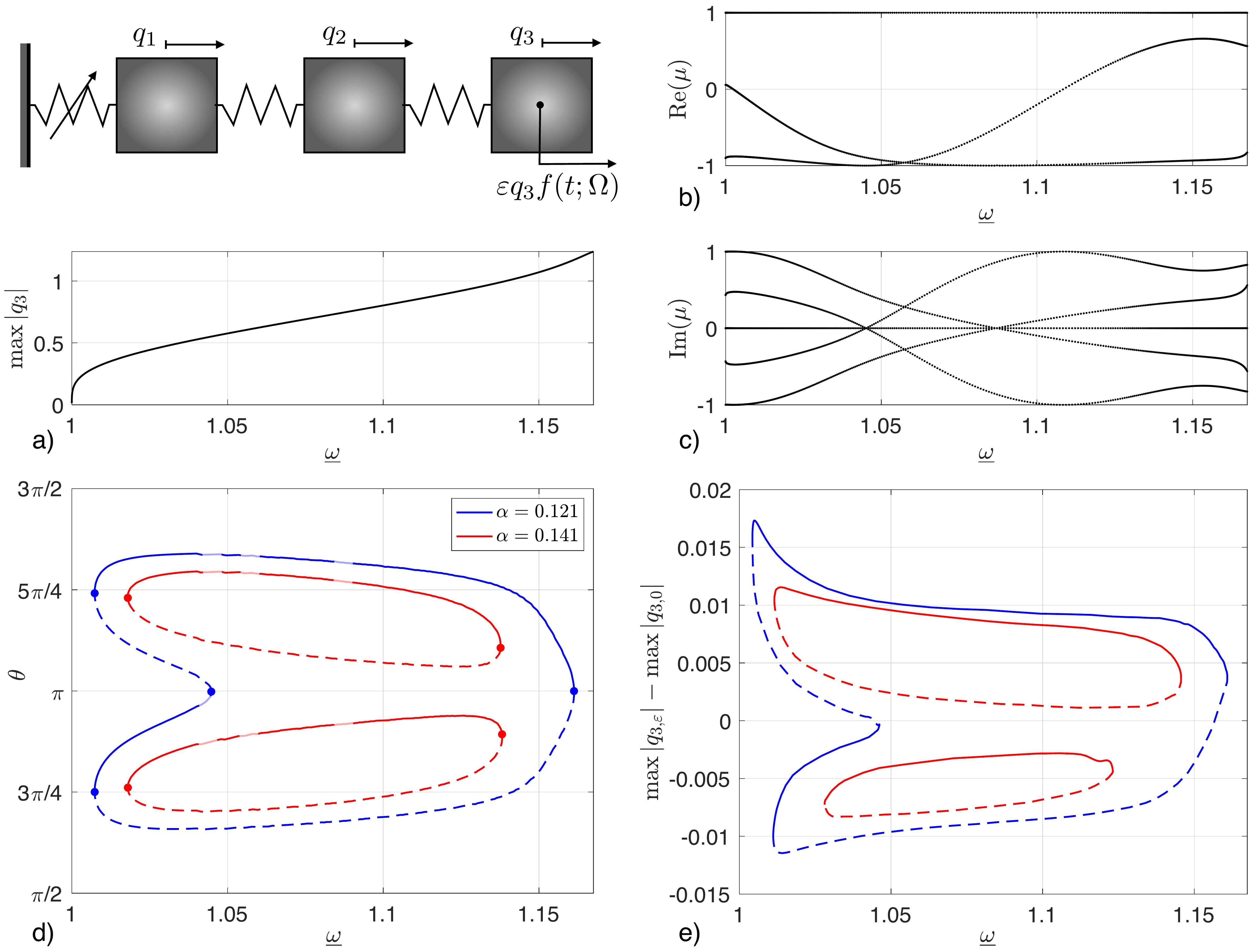}
	\caption{Plot (a) shows the backbone curve in terms of the normalized frequency $\underline{\omega}$ and the amplitude of the first family of periodic solutions for the conservative limit of system (\ref{eq:systS52}). The trend of Floquet multipliers of the family is shown in plot (b) in terms of their real part, while plot (c) regards the imaginary one. Plot (d) illustrates the level set of the Melnikov function for two values of the damping in the plane $(\underline{\omega},,\theta)$, where the latter is the orbit phase. Here, dashed lines predict unstable perturbed orbits, while solid ones indicate asymptotically stable ones. The latter feature faded regions at which the predictions of Theorem \ref{thm:astability} could turn out to be weak. The dots instead denote saddle-node bifurcations. Plot (e) shows the numerical simulations of the frequency response for the perturbed system at $\varepsilon=0.0025$. The horizontal axis displays the normalized frequency and the vertical one the distance from the conservative backbone curve.}
	\label{fig:S5_IM2}
\end{figure}

For the conservative limit of the system, the origin is the unique fixed point\footnote{The origin is the only point at which the potential is stationary. This condition is expressed by t{}he following equations: \begin{equation*} q_3=q_2,  \,\,\,\,\, q_2=q_1,  \,\,\,\,\, q_1(q_1^2-q_1/2+1/3)=0. \end{equation*}} and no resonances occur among its linearized frequencies $(0.30394,\,1.0854,\,1.7501)$. Thus, three families of periodic orbits emanate from the origin \cite{Lyapunov1992}, and they can be parametrized with the value of the first integral $h$. We focus on detecting perturbed solutions arising from the lowest-frequency family. We denote by $\underline{\omega}(h)$ the frequency of the periodic orbits in this family normalized by the linear limit at $h=0$, i.e. $\underline{\omega}(h)=T(0)/T(h)$. When performing numerical continuation, the first family is 1-normal and satisfies assumptions (A.2) and (A.3) of section \ref{sec:S3} for $1<\underline{\omega}(h)<1.165$, which is our frequency range of interest. The backbone curve for this family is illustrated in Fig. \ref{fig:S5_IM2}a in terms of the normalized frequency and the maximum amplitude of the coordinate $q_3$ along periodic orbits, denoted by $\max |q_3|$. This backbone curve displays a hardening trend, i.e., $\underline{\omega}'(h)>0$, $T'(h)<0$. Figures \ref{fig:S5_IM2}b and \ref{fig:S5_IM2}c respectively show the trend of real and imaginary parts of the Floquet multipliers of the family.  

We select a perturbation in Eq. (\ref{eq:systS52}) to satisfy the assumptions of Proposition \ref{prop:uniformcontraction}, so that the volume contractions $C_\mathcal{V}$ are always positive. Moreover, the forcing corresponds to the sixth harmonic approximation of a square wave with unit amplitude and period $2\pi/\Omega$. We now examine via our Melnikov approach perturbed periodic orbits when the forcing period is in $1:1$ resonance with the period of the orbits of the first family, thus we set $\Omega=2\pi/T(h)$. Sweeping through the family, we evaluate the Melnikov function on every orbit and hence construct a scalar function $M_{1:1}(\theta,h,\alpha)$, using the phase $\theta = 2\pi s / T(h)$ instead of the shift $s$.

Figure \ref{fig:S5_IM2}d shows the zero level set of $M_{1:1}(\theta,h,0.121)$, in blue, and of $M_{1:1}(\theta,h,0.141)$, in red, plotted in the plane $(\theta,\underline{\omega})$. Solid lines indicate zeros in $\theta$ with $\partial_\theta M_{1:1} < 0$, while dashed ones indicate zeros with $\partial_\theta M_{1:1} > 0$. According to Theorem \ref{thm:instability}, the latter zeros predict unstable perturbed periodic orbits, while the former predict asymptotically stable ones Theorem \ref{thm:astability}. However, there are conservative orbits of the family that either feature a pair of Floquet multipliers related to the normal space equal to $-1$, or have two coincident complex conjugated pairs of Floquet multipliers (cf. Figs. \ref{fig:S5_IM2}b,c). At these resonances, Theorem \ref{thm:astability} is not applicable and hence the prediction of asymptotic stability could fail in the vicinity of these orbits, shown with faded solid lines in Fig. \ref{fig:S5_IM2}d.  Moreover, the dots in Fig. \ref{fig:S5_IM2}d depict quadratic zeros with respect to $\theta$ of $M_{1:1}$ at which saddle-node bifurcations occur.

From lower to higher frequencies for $\alpha=0.121$, the Melnikov has two quadratic zeros when $\Omega\approx1.01$. Aferwards, they evolve as four simple zeros, then the internal pair collapses in a quadratic zero and the remaining zeros persist until $\Omega(h)\approx 1.158$. For this value of $\alpha$, the multiharmonic, parametric forcing of Eq. (\ref{eq:systS52}) generates an inverse cup-shaped isolated response curve. In contrast, two disjoint isolas exist when damping is increased at $\alpha=0.141$. These predictions are confirmed via direct numerical simulations of the perturbed system in Fig. \ref{fig:S5_IM2}e for $\varepsilon=0.0025$, plotted with corresponding colors. Here we depict the frequency response by showing the distance, in terms of the amplitude $\max |q_3|$, to the conservative (isochronous) limit. When interpreting these results, one has to recall from \cite{Cenedese2019} that the Melnikov function is the leading-order approximation of the bifurcation function governing the persistence problem. Hence, the level sets in Fig. \ref{fig:S5_IM2}b are approximate ones and, in particular, the symmetric appearance of zeros will be, generically, destroyed in the full bifurcation function.

\section{Conclusion}
We have developed an analytical approach to determine the stability of forced-damped oscillations of nonlinear, multi-degree-of-freedom mechanical systems. Specifically, by studying these motions as perturbations from conservative backbone curves, we have complemented the existence results of \cite{Cenedese2019} with further results on the stability of the periodic forced-damped response. Other than the Melnikov function (which also predicts persistence), the frequency variation within the limiting conservative periodic orbits and the nonlinear damping rates (or volume contractions) play a role in the assessment of stability. These damping rates provide estimates for the Floquet multipliers of the forced-damped response, thereby predicting stability of perturbed trajectories from their conservative limit. 

After proving the method in a general setting, we verified our analytical predictions on two specific examples. In the first, we studied subharmonic resonances with external forcing in a gyroscopic two-degree-of-freedom system considering two damping mechanisms. Even though these latter dissipate the same amount of energy along the conservative limit, their stability indicators (and potentially their basins of attraction) are different. In the second example, we considered the case of parametric forcing on a three-degree-of-freedom oscillator with mass-proportional damping. By considering multi-harmonic excitation, we have successfully described the generation and the stability of periodic trajectories that lie on exotic isolas of the frequency response.

When the mechanical system has only one degree of freedom, the conditions we derived coincide with prior analyses of \cite{GH1983,Yagasaki1996}. Our results are also consistent with numerical and experimental observations reported in available studies \cite{Touze2006,Peeters2011a,Peeters2011b,Ehrhardt2016,Renson2016b,Hill2017,Szalai2017,Kovacic2008,Avramov2008,Ponsioen2019,Vakakis1997,Vakakis2001,Vakakis2008,Avramov2011,Avramov2013,Kerschen2014}, in that they are able to explain hysteresis of frequency responses in mechanical systems as well as other stability or bifurcation phenomena. From a numerical perspective, being able to make predictions based on the conservative limit alone can lead to major savings in computational time when analyzing weakly-damped systems with a large number of degrees of freedom. Moreover, our analytic stability criteria may help in overcoming potential issues with numerical methods, such as computational complexity and convergence problems.

\begin{appendix}
\section{Proofs of the main results}
\label{app:A1}
In this section, we adopt the assumptions of section \ref{sec:S4} and derive an approximation for the Floquet multipliers of the non-autonomous periodic orbit $\mathcal{Z}_{\varepsilon}$, to be used for the stability assessment \cite{GH1983,Chicone2000,Teschl2012}. This orbit solves system (\ref{eq:NAutSysG}) with $\delta(\varepsilon)=m\tau/l+\varepsilon \tilde{\tau} +o(\varepsilon)$ and has initial condition $\xi(\varepsilon)=z+\varepsilon \tilde{z}+o(\varepsilon)$.

We aim at analyzing the evolution of the Floquet multipliers of $\mathcal{Z}_{\varepsilon}$ from those of its conservative limit $\mathcal{Z}$. Since we rely on the linearized flow and look at the $O(\varepsilon)$-perturbation as in the next lemma, our stability assessment is valid in a small neighborhood of $z$ and for $\varepsilon$ small enough.
\begin{lemma}
\label{lem:pertexp}
The monodromy matrix $P(\varepsilon)=X(l\delta(\varepsilon);\xi(\varepsilon),T(\varepsilon),\varepsilon)$ of $\mathcal{Z}_{\varepsilon}$ is approximated at order $O(\varepsilon^2)$ by the following expansion
\begin{equation}
\label{eq:monZeexp}
P(\varepsilon)=\Pi+\varepsilon \Pi\left(\Psi^{m:l}_H+\Psi^{m:l}_g\right)+O(\varepsilon^2),
\end{equation}
where
\begin{equation}
\label{eq:FVareSol}
\begin{split}
\Psi^{m:l}_H &= \displaystyle \int_0^{m\tau}X^{-1}_0(t;z)JD^3H(x_0(t;z),x_1(t;z,\tilde{z}))X_0(t;z)dt, \\ \\
x_1(t;z,\tilde{z}) &=\displaystyle X_0(t;z)\tilde{z}+X_0(t;z)\int_0^{t}X^{-1}_0(s;z)g(x_0(s;z),s;m\tau/l,0) ds, \\ \\
\Psi^{m:l}_g &=\displaystyle \int_0^{m\tau}X^{-1}_0(t;z)\partial_x g(x_0(t;z),t;m\tau/l,0)X_0(t;z)dt.
\end{split}
\end{equation}
\end{lemma}
\begin{proof}
The smoothness assumption for the vector fields involved are sufficient to approximate the solutions of systems (\ref{eq:NAutSysG}) and (\ref{eq:NAutSysG_FV}) at $O(\varepsilon^2)$. The periodic orbit $\mathcal{Z}_{\varepsilon}$ is approximated by $x_0(t,z)+\varepsilon x_1(t;z,\tilde{z})$, where $x_1(t;z,\tilde{z})$ solves the initial value problem  
\begin{equation}
\label{eq:FVare}
\begin{array}{lr} 
\dot{x}_1=JD^2H(x_0(t;z))x_1+g(x_0(t;z),t;m\tau/l,0),& x_1(0)=x_1(m\tau)=\tilde{z}.
\end{array}
\end{equation}
By substituting $x=x_0+\varepsilon x_1$, $X=X_0+\varepsilon X_1$ and $\delta=m\tau/l+O(\varepsilon)$ in system (\ref{eq:NAutSysG_FV}) and expanding the equation, one obtains at $O(\varepsilon)$ the following initial value problem
\begin{equation}
\begin{array}{l} 
\dot{X}_1=JD^2H(x_0(t;z))X_1+\left(JD^3H(x_0(t;z),x_1(t;z,\tilde{z})) +\partial_x g(x_0(t;z),t;m\tau/l,0) \right)X_0(t;z),  \\ X_1(0)=0,
\end{array}
\end{equation}
whose analytical solution, expressed by Lagrange's formula \cite{Teschl2012}, at time $m\tau$ is the $O(\varepsilon)$-term in Eq. (\ref{eq:monZeexp}).
\end{proof}

\subsection{Some results in linear algebra}
\label{app:A1S1}
We now introduce some factorization results exploiting the properties of the symplectic group, denoted $\mathrm{Sp}(2n,\mathbb{R})$. The next lemma characterizes the relation between strongly invariant subspaces related to $\Pi$, as defined in section \ref{sec:S4}.
\begin{lemma}
\label{lem:Jorth}
Let $\Pi \in \mathrm{Sp}(2n,\mathbb{R})$ and assume that $\Pi$ has two distinct strongly invariant subspaces $\mathcal{V}$ and $\mathcal{W}$ of dimensions $2v$ and $2w$, respectively. Let these subspaces be represented by $R_\mathcal{V}\in\mathbb{R}^{2n\times2v}$ and $R_\mathcal{W}\in\mathbb{R}^{2n\times2w}$, respectively, so that, for a unique $B_\mathcal{V}\in\mathbb{R}^{2v\times2v}$ and $B_\mathcal{W}\in\mathbb{R}^{2w\times2w}$, the following identities hold:
\begin{equation}
\begin{array}{ll}
\Pi R_\mathcal{V} = R_\mathcal{V}B_\mathcal{V} ,& \Pi R_\mathcal{W} = R_\mathcal{W}B_\mathcal{W}.
\end{array}
\end{equation}
Then:
\begin{enumerate}[label=\textit{(\roman*)}]
\item $ R_\mathcal{V}^{\top} J R_\mathcal{W}=0$ and $ R_\mathcal{W}^{\top} J R_\mathcal{V}=0$,
\item $R_\mathcal{V}^{\top} J R_\mathcal{V} $ and $ R_\mathcal{W}^{\top} J R_\mathcal{W}$ are invertible.
\end{enumerate}
\end{lemma}
\begin{proof}
We first prove statement \textit{(i)}. Recalling the standard symplectic identity $\Pi^\top J\Pi = J$, we find
\begin{equation}
\label{eq:pJorth1}
R_\mathcal{V}^\top J R_\mathcal{W} = (\Pi R_\mathcal{V}B_\mathcal{V}^{-1})^\top\Pi R_\mathcal{W}B_\mathcal{W}^{-1} = B_\mathcal{V}^{-\top} Rv^\top \Pi^\top J\Pi R_\mathcal{W}B_\mathcal{W}^{-1} = B_\mathcal{V}^{-\top} Rv^\top J R_\mathcal{W}B_\mathcal{W}^{-1}.
\end{equation}
By denoting $A=R_\mathcal{V}^\top J R_\mathcal{W}$ and considering the leftmost and the rightmost sides of Eq. (\ref{eq:pJorth1}), we obtain a homogeneous Sylvester equation
\begin{equation}
\label{eq:sylveq}
B_\mathcal{V}^{\top} A = A B_\mathcal{W}^{-1}.
\end{equation}
The eigenvalues of $B_\mathcal{V}^\top$ are equal to those of $B_\mathcal{V}$. Since $\Pi$ is symplectic and by Definition \ref{def:stronginv}, $B_\mathcal{W}^{-1}$ and $B_\mathcal{W}$ have identical eigenvalues as well. By assumption, the eigenvalues of $B_\mathcal{V}$ and $B_\mathcal{W}$ are distinct. Hence, $B_\mathcal{V}^{\top}$ and $B_\mathcal{W}^{-1}$ have no common eigenvalue, which implies that Eq. (\ref{eq:sylveq}) has the unique solution $A=R_\mathcal{V}^\top J R_\mathcal{W} =0$ (see Chapter VIII in \cite{Gantmacher1959}). Transposing $A$ and using the identity $J^\top=-J$, one also gets that $R_\mathcal{W}^\top J R_\mathcal{V} =0$.

We now prove statement \textit{(ii)}. Since $\mathcal{V}$ and $\mathcal{W}$ correspond to the direct sums of spectral subspaces (sometimes called root spaces) of distinct eigenvalues of $\Pi$, we have that, by construction, $\mathcal{V}\cap \mathcal{W} = \emptyset$ (see Theorem 2.1.2 in \cite{Lancaster2006} for a proof). By letting $\mathcal{Y}=\mathcal{V}\oplus \mathcal{W}$, we then define the linear map $A_\mathcal{V} := R_\mathcal{V}^\top J$ and we analyze its restriction to $\mathcal{Y}$, i.e. $A_\mathcal{V}|_\mathcal{Y}:\mathcal{Y}\rightarrow \mathbb{R}^{2v}$. Since the kernel of this map is $\mathcal{W}$, then its image must have dimension $\mathrm{dim}(\mathcal{Y})-\mathrm{dim}(\mathcal{W}) = 2v$ by the rank-nullity theorem. Hence, $A_\mathcal{V}R_\mathcal{V} = R_\mathcal{V}^\top J R_\mathcal{V}$ is invertible. An analogous reasoning holds for the linear map $A_\mathcal{W} := R_\mathcal{W}^\top J$.
\end{proof}
The next result follows as a consequence of Lemma \ref{lem:Jorth}.
\begin{lemma}
\label{lem:Jfact}
Let $\Pi \in \mathrm{Sp}(2n,\mathbb{R})$ and assume $\Pi$ has two distinct strongly invariant subspaces $\mathcal{V}$ and $\mathcal{W}$ be such that $\mathcal{V}\oplus\mathcal{W}=\mathbb{R}^{2n}$. Denote by $2v$ the dimension of $\mathcal{V}$ and let these subspaces be spanned by a linear combination of the columns in $R_\mathcal{V}\in\mathbb{R}^{2n\times2v}$ and $R_\mathcal{W}\in\mathbb{R}^{2n\times2(n-v)}$, respectively. Define the symplectic left inverse matrices for $R_\mathcal{V}$ and $R_\mathcal{W}$ respectively as
\begin{equation}
\begin{array}{lr}
S_\mathcal{V} = (R_\mathcal{V}^\top J R_\mathcal{V})^{-1}R_\mathcal{V}^\top J , &  S_\mathcal{W} = (R_\mathcal{W}^\top J R_\mathcal{W})^{-1}R_\mathcal{W}^\top J.
\end{array}
\end{equation}
Then, the following factorization holds
\begin{equation}
\label{eq:Jfact}
\begin{array}{lccr}
R^{-1} = \begin{bmatrix} S_\mathcal{V} \\ S_\mathcal{W} \end{bmatrix} , & R = \begin{bmatrix} R_\mathcal{V} & R_\mathcal{W} \end{bmatrix} 
, & R^{-1} \Pi R = \begin{bmatrix} B_\mathcal{V}& 0  \\ 0 & B_\mathcal{W} \end{bmatrix}, 
\end{array}
\end{equation}
where the eigenvalues of $B_\mathcal{V}\in\mathbb{R}^{2v\times2v}$ are distinct from those of $B_\mathcal{W}\in\mathbb{R}^{2(n-v)\times2(n-v)}$.
\end{lemma}
We also remark that one can obtain from Eq. (\ref{eq:Jfact}) the further identities
\begin{equation}
\label{eq:reigPi}
\begin{array}{lr}
S_\mathcal{V}\Pi = B_\mathcal{V}S_\mathcal{V} ,& S_\mathcal{W}\Pi = B_\mathcal{W}S_\mathcal{W}.
\end{array}
\end{equation}
We will also need the next result. 
\begin{lemma}
\label{lem:Jtrace}
Let $A\in\mathbb{R}^{2n \times 2n}$ and let $\mathcal{V}$ be a strongly invariant subspace for $\Pi$. Let the columns of $R_\mathcal{V}\in\mathbb{R}^{2n\times 2v}$ be a basis for $\mathcal{V}$ and let $S_\mathcal{V}$ be the symplectic left inverse of $R_{\mathcal{V}}$. Then
\begin{equation}
\label{eq:traceeq}
\mathrm{trace}(S_\mathcal{V} A R_\mathcal{V}) = \frac{1}{2} \mathrm{trace}\left(S_\mathcal{V} (A-JA^\top J) R_\mathcal{V}\right).
\end{equation}
In particular, if $A=J\hat{A}$ where $\hat{A}$ is symmetric, then $\mathrm{trace}(S_\mathcal{V} A R_\mathcal{V})=0$.
\end{lemma}
\begin{proof}
Recall that $J^\top J =I$ and that, for arbitrary matrices $A_1$ and $A_2$, $\mathrm{trace}(A_1A_2)=\mathrm{trace}(A_2A_1)$ as well as $\mathrm{trace}(A_1)=\mathrm{trace}(A_1^\top)$. Equation (\ref{eq:traceeq}) holds because
\begin{equation}
\begin{split}
\mathrm{trace}(S_\mathcal{V} A R_\mathcal{V})&\displaystyle=  \mathrm{trace}\left((R^\top_\mathcal{V} J R_\mathcal{V})^{-1}R^\top_\mathcal{V} J A R_\mathcal{V}\right)= \mathrm{trace}\left(((R^\top_\mathcal{V} J R_\mathcal{V})^{-1}R^\top_\mathcal{V} JA R_\mathcal{V})^\top\right) 
\\
&\displaystyle= \mathrm{trace}\left(R^\top_\mathcal{V} A^\top J^\top R_\mathcal{V}(R^\top_\mathcal{V} J R_\mathcal{V})^{-\top}\right)= \mathrm{trace}\left(R^\top_\mathcal{V} A^\top J^\top R_\mathcal{V}((R^\top_\mathcal{V} J R_\mathcal{V})^\top)^{-1}\right) 
\\
&\displaystyle= \mathrm{trace}\left(R^\top_\mathcal{V} A^\top J^\top R_\mathcal{V}(R^\top_\mathcal{V} J^\top R_\mathcal{V})^{-1}\right) =\mathrm{trace}\left((R^\top_\mathcal{V} J^\top R_\mathcal{V})^{-1}R^\top_\mathcal{V} A^\top J^\top R_\mathcal{V}\right) 
\\
&\displaystyle= \mathrm{trace}\left((R^\top_\mathcal{V} J R_\mathcal{V})^{-1}R^\top_\mathcal{V} A^\top J R_\mathcal{V}\right)  =\mathrm{trace}\left((R^\top_\mathcal{V} J R_\mathcal{V})^{-1}R^\top_\mathcal{V} J J^\top A^\top J R_\mathcal{V}\right) 
\\
&=\mathrm{trace}(S_\mathcal{V} J^\top A^\top J R_\mathcal{V})=-\mathrm{trace}(S_\mathcal{V} J A^\top J R_\mathcal{V}).
\end{split}
\end{equation}
The last statement can be found by direct substitution of $A=J\hat{A}$ in Eq. (\ref{eq:traceeq}).
\end{proof}
\subsection{Perturbation of the Floquet multipliers}
\label{app:A1S2}
We now apply the results in section \ref{app:A1S1} to the initial perturbation expansion and we use the following definition from \cite{Karow2014}.
\begin{definition}
Let $v,w$ be aribitrary integers, $B_v\in\mathbb{R}^{v\times v}$ and $B_w\in\mathbb{R}^{w\times w}$. The separation of $B_v$ and $B_w$ with respect to an arbitrary norm $||\cdot||$ is defined as
\begin{equation}
\mathrm{sep}(B_v,B_w):=\min_{Y\in\mathbb{R}^{w\times v}:||Y||=1}||YB_v-B_wY||.
\end{equation}
\end{definition}
We remark that $\mathrm{sep}(B_v,B_w)\neq 0$ if and only if the eigenvalues of $B_v$ are different from those of $B_w$. We can then state the next fundamental result.
\begin{theorem}
\label{thm:eigSplit}
Consider the perturbation expansion of Lemma \ref{lem:pertexp} and the setting of Lemma \ref{lem:Jfact}. If $\varepsilon << \mathrm{sep}(B_\mathcal{V},B_\mathcal{W})$, then the eigenvalues of $P(\varepsilon)$ coincide with the eigenvalues of the matrices
\begin{equation}
\begin{array}{lr}
B_\mathcal{V}\left( I+\varepsilon S_\mathcal{V} (\Psi^{m:l}_H+\Psi^{m:l}_g) R_\mathcal{V}\right)+ O(\varepsilon^2),& B_\mathcal{W}\left(I+\varepsilon S_\mathcal{W} (\Psi^{m:l}_H+\Psi^{m:l}_g) R_\mathcal{W}\right)+ O(\varepsilon^2).
\end{array}
\end{equation}
\end{theorem}
\begin{proof}
We use the shorthand notation $\Psi = \Psi^{m:l}_H+\Psi^{m:l}_g$. We then define the matrix
\begin{equation}
A(\varepsilon)=R^{-1} P(\varepsilon) R = R^{-1} \Pi R + \varepsilon R^{-1} \Pi \Psi R + O(\varepsilon^2)=\begin{bmatrix} B_\mathcal{V}(I+\varepsilon \Psi _{\mathcal{V}}) & B_\mathcal{V}\varepsilon \Psi _{\mathcal{V}\mathcal{W}}  \\ \varepsilon B_\mathcal{W}\Psi _{\mathcal{W}\mathcal{V}} & B_\mathcal{W}(I+\varepsilon \Psi _{\mathcal{W}}) \end{bmatrix}+O(\varepsilon^2), 
\end{equation}
where, exploiting the identities in Eq. (\ref{eq:reigPi}), we used the notation
\begin{equation}
\begin{array}{lccr}
\Psi_\mathcal{V}=S_\mathcal{V} \Psi R_\mathcal{V},& \Psi_\mathcal{VW}=S_\mathcal{V} \Psi R_\mathcal{W}, & \Psi_\mathcal{WV}=S_\mathcal{W} \Psi R_\mathcal{V}, & \Psi_\mathcal{W}=S_\mathcal{W} \Psi R_\mathcal{W}.
\end{array}
\end{equation}
Note that, by similarity, $A(\varepsilon)$ has the same spectrum as $P(\varepsilon)$. Since the eigenvalues of $B_\mathcal{V}$ are different from those of $B_\mathcal{W}$, Corollary 2.4 in \cite{Karow2014} guarantees, for $\varepsilon$ small enough, the existence of a strongly invariant subspace for $A(\varepsilon)$ whose coordinates can be described by the asymptotic expansion
\begin{equation}
V(\varepsilon)=\begin{bmatrix} I  \\ \displaystyle \varepsilon\left(\mathrm{sep}(B_\mathcal{V},B_\mathcal{W}) \right)^{-1} W_V(\varepsilon) \end{bmatrix},
\end{equation}
which is justified if $\varepsilon << \mathrm{sep}(B_\mathcal{V},B_\mathcal{W})$ and for a unique $W_V:\mathbb{R}\rightarrow\mathbb{R}^{2 (n-v) \times 2 v}$. This result holds as a consequence of the implicit function theorem. By similarity, the strongly invariant subspace $\mathcal{V}$ for $P(0)=\Pi$ persist as $\mathcal{V}_\varepsilon$ for $P(\varepsilon)$ and $\mathcal{V}_\varepsilon$ is described by the columns of the product $RV(\varepsilon)$. Then, the invariance relation $A(\varepsilon)V(\varepsilon)= V(\varepsilon) B_V(\varepsilon)$ holds for a unique $B_V:\mathbb{R}\rightarrow\mathbb{R}^{2 v \times 2 v}$ whose eigenvalues are the ones related to $\mathcal{V}_\varepsilon$. By using this invariance relation, one obtains
\begin{equation}
\begin{bmatrix} B_\mathcal{V}(I+\varepsilon \Psi _{\mathcal{V}})  \\ \varepsilon B_\mathcal{W}\Psi _{\mathcal{W}\mathcal{V}} \end{bmatrix} + O(\varepsilon^2)=\begin{bmatrix} B_V(\varepsilon)  \\ \varepsilon \left(\mathrm{sep}(B_\mathcal{V},B_\mathcal{W}) \right)^{-1} W_V(\varepsilon) B_V(\varepsilon) \end{bmatrix}.
\end{equation}
An analogous discussion also applies to $\mathcal{W}$, so the claim is proved.
\end{proof}
This result always holds asymptotically, but we have highlighted the fact that $\varepsilon$ should stay below a critical threshold which guarantees that the eigenvalues related to $\mathcal{V}_\varepsilon$ and to $\mathcal{W}_\varepsilon$ remain separated. The references \cite{Stewart1990,Karow2014} establish non-asymptotic bounds for this kind of perturbations.

Next, we show an important consequence of Lemma \ref{lem:Jtrace}.
\begin{lemma}
\label{lem:productperteigV}
Let $\mathcal{V}$ be a strongly invariant subspace for $\Pi$, let the columns of $R_\mathcal{V}\in\mathbb{R}^{2n\times 2v}$ span $\mathcal{V}$ and let $S_\mathcal{V}$ be the symplectic left inverse of $R_{\mathcal{V}}$. Then, we have
\begin{equation}
\label{eq:productperteigV}
\mathrm{det}\left( B_\mathcal{V}\left( I+\varepsilon S_\mathcal{V} (\Psi^{m:l}_H+\Psi^{m:l}_g) R_\mathcal{V}\right) + O(\varepsilon^2)\right) = 1 - \varepsilon m\tau v C_\mathcal{V} + O(\varepsilon^2)
\end{equation}
where the value $C_\mathcal{V}$ is defined in Eq. (\ref{eq:contrdef}).
\end{lemma}
\begin{proof}
We recall that $\mathrm{det}( B_\mathcal{V})=1$ by definition and the identity $D_\varepsilon\mathrm{det}(A(\varepsilon))=\mathrm{trace}\left(A^*(\varepsilon)D_\varepsilon A(\varepsilon)\right)$, also called Jacobi formula, which holds for some smooth matrix family $A(\varepsilon)$ with $A^*(\varepsilon)$ being the adjugated matrix of $A(\varepsilon)$. Since the determinant of a product of square matrices is equal to the product of their determinants, one obtains the Taylor expansion
\begin{equation}
\begin{split}
\mathrm{det}\left( B_\mathcal{V}\left( I+\varepsilon S_\mathcal{V} (\Psi^{m:l}_H+\Psi^{m:l}_g) R_\mathcal{V}\right) + O(\varepsilon^2)\right) & =\mathrm{det}\left( I+\varepsilon S_\mathcal{V} (\Psi^{m:l}_H+\Psi^{m:l}_g) R_\mathcal{V}+ O(\varepsilon^2)\right) \\ &= 1 + \varepsilon \mathrm{trace}\left(S_\mathcal{V} (\Psi^{m:l}_H+\Psi^{m:l}_g) R_\mathcal{V}\right)+ O(\varepsilon^2).
 \end{split}
\end{equation}
By linearity, the $O(\varepsilon)$-term above can be split in $\mathrm{trace}\left(S_\mathcal{V}\Psi^{m:l}_H R_\mathcal{V}\right)+\mathrm{trace}\left(S_\mathcal{V}\Psi^{m:l}_g R_\mathcal{V}\right)$ and we now prove that the first of these traces vanishes. Indeed, the third derivative in $\Psi_H^{m:l}$ can be expressed as
\begin{equation}
JD^3H(x_0(t;z),x_1(t;z,\tilde{z}))=\sum_{k=1}^n \frac{\partial }{\partial x_k}\Big(JD^2H(x) \Big) \Bigr|_{\substack{x=x_0(t;z)}}x_{k,1}(t;z,\tilde{z})=\sum_{k=1}^n J A_k(t)x_{k,1}(t;z,\tilde{z}),
\end{equation}
where the scalars $x_{k,1}(t;z,\tilde{z})$ identify the components of the curve $x_1(t;z,\tilde{z})$ and the matrix families $A_k(t)$ are symmetric. Recalling the identity $X^{-1}_0(t;z) = JX^\top_0(t;z)J^\top $ \cite{deGosson2011}, we can therefore write 
\begin{equation}
\begin{split}
\Psi^{m:l}_H &\displaystyle=\int_0^{m\tau} \sum_{k=1}^n X^{-1}_0(t;z) JA_k(t) X_0(t;z)x_{k,1}(t;z,\tilde{z})dt \\&\displaystyle=J\int_0^{m\tau} \sum_{k=1}^n X^\top_0(t;z) A_k(t) X_0(t;z)x_{k,1}(t;z,\tilde{z})dt = J\hat{A},
\end{split}
\end{equation}
where $\hat{A}$ is still symmetric. Thus, by Lemma \ref{lem:Jtrace}, $\mathrm{trace}\left(S_\mathcal{V}\Psi^{m:l}_H R_\mathcal{V}\right)=0$ and, by linearity, 
\begin{equation}
\begin{split}
\mathrm{trace}\left(S_\mathcal{V}\Psi^{m:l}_g R_\mathcal{V}\right) &\displaystyle = \mathrm{trace}\left(S_\mathcal{V}\int_0^{m\tau} X^{-1}_0(t;z)\partial_x g(x_0(t;z),t;m\tau/l,0)X_0(t;z) dt \,R_\mathcal{V}\right) \\  &\displaystyle = -m\tau v C_\mathcal{V}.
\end{split}
\end{equation}
\end{proof}

\subsection{Proof of the main theorems}
\label{app:A1S3}
\begin{proof}[Proof of Theorem \ref{thm:instability}] We will use the shorthand notation $\Psi = \Psi^{m:l}_H+\Psi^{m:l}_g$. Let us start with statement \textit{(ii)}. For $\varepsilon$ sufficiently small, Theorem \ref{thm:eigSplit} guarantees that, to unfold the evolution of some Floquet multipliers, it is sufficient to study the $2v$ eigenvalues $\mu_{k}(\varepsilon)$ of $B_\mathcal{V}\left( I+\varepsilon S_\mathcal{V} \Psi R_\mathcal{V} \right)+ O(\varepsilon^2)$ related to some unperturbed, strongly invariant subspace $\mathcal{V}$. According to Lemma \ref{lem:productperteigV} and due to the fact that the determinant is equal to the product of the eigenvalues, we have
\begin{equation}
\prod_{k=1}^{2v} \mu_{k}(\varepsilon) = \mathrm{det}\left( B_\mathcal{V}\left( I+\varepsilon S_\mathcal{V} \Psi R_\mathcal{V}\right) + O(\varepsilon^2)\right) = 1 -\varepsilon m\tau v C_\mathcal{V}+ O(\varepsilon^2).
\end{equation}
Thus, for $\varepsilon > 0$ small enough and $C_\mathcal{V} < 0$, there exist at least one index $\tilde{k}$ for which $|\mu_{\tilde{k}}(\varepsilon)|>1$, that implies instability.

We then prove statement \textit{(i)}, for which the discussion is more involved. As a basis for the tangent space $\mathcal{T}$, we choose the vector field $JDH(z)$ and a vector $b(z)\in\mathcal{T}$ orthogonal to $JDH(z)$ and normalized such that $\langle b(z) , DH(z) \rangle=1$. From Proposition A.1 in \cite{Cenedese2019}, the following identities hold 
\begin{equation}
\label{eq:rel1eig}
\begin{array}{lr}
\Pi^m(z)JDH(z)=JDH(z) ,& \Pi^m(z) b(z)=b(z)-mT'(h)JDH(z),
\end{array}
\end{equation}
where $h=H(z)$. Denoting $R_\mathcal{T}=[JDH(z) \,\,\, b(z)]$ and $a =mT'(h)$, from Eq. (\ref{eq:rel1eig}) the symplectic left inverse of $R_\mathcal{T}$ and their relative block $B_\mathcal{T}$ read 
\begin{equation}
\begin{array}{lr}
 S_\mathcal{T} =\begin{bmatrix} Jb(z) & DH(z) \end{bmatrix}^\top, & B_\mathcal{T}=\begin{bmatrix} 1 & -a \\ 0 & 1 \end{bmatrix}.
\end{array}
\end{equation}
According to Theorem \ref{thm:eigSplit}, to evaluate the perturbation of these coincident Floquet multipliers of the conservative limit, we need to study the eigenvalues $(\mu_1(\varepsilon),\mu_2(\varepsilon))$ of the perturbed block
\begin{equation}
B_\mathcal{T}\left( I+\varepsilon S_\mathcal{T} \Psi R_\mathcal{T}\right) + O(\varepsilon^2)=\begin{bmatrix} 1+\varepsilon (a_{11}-a a_{21}) & -a +\varepsilon (a_{12}-a a_{22})\\ \varepsilon a_{21} & 1+\varepsilon a_{22} \end{bmatrix} + O(\varepsilon^2),
\end{equation}
where we denoted $a_{jk}$ the components of the matrix $S_\mathcal{V} \Psi R_\mathcal{V}$. These eigenvalues are expressed as
\begin{equation}
\lambda_{1,2}=1+\varepsilon \frac{a_{11}+a_{22}-a a_{21}}{2}\pm\sqrt{\varepsilon}\sqrt{-a a_{21}} + o(\varepsilon).
\end{equation}
The value $a a_{21} = mT'(h) a_{21}$ acts as discriminant and we later prove that $a_{21}=M_{m:l}'(0)$. Thus, if $T'(h) M_{m:l}'(0)<0$ and $\varepsilon > 0$, then there exists a real eigenvalue greater than $1$, which implies instability as in statement \textit{(i)} of the Theorem for small $\varepsilon>0$. Conversely, the two eigenvalues are complex conjugated for small $\varepsilon >0$ whose squared modulus, using Lemma \ref{lem:productperteigV}, reads
\begin{equation}
\label{eq:modFMT}
\mu_1(\varepsilon)\mu_2(\varepsilon)=|\mu_1(\varepsilon)|^2=\mathrm{det}\left( B_\mathcal{T}\left( I+\varepsilon S_\mathcal{T} \Psi R_\mathcal{T}\right) + O(\varepsilon^2)\right) =1-\varepsilon m\tau C_\mathcal{T} + O(\varepsilon^2).
\end{equation}
If $C_{\mathcal{T}}<0$, the two eigenvalues evolve as a complex conjugated couple outside the unit circle in the complex plane for small $\varepsilon>0$, proving then statement \textit{(i)}.

We now show that $a_{21}=\langle DH(z),\Psi JDH(z) \rangle=DM_{m:l}(0)=M_{m:l}'(0)$. We first recall the following identities (see \cite{Cenedese2019,Chicone2000} for a proof):
\begin{equation}
\label{eq:solFvarCL}
\begin{array}{lr}
X_0(t;z)JDH(z)=JDH(x_0(t;z)) , & DH(z)X^{-1}_0(t;z)=DH(x_0(t;z)). 
\end{array}
\end{equation}
We also recall the fact that the gradient of the first integral solves the adjoint variational equation of the conservative limit (see Proposition 3.2 in \cite{Li2000} for a proof), i.e.,
\begin{equation}
\label{eq:DHFVadj}
D_t\big(DH(x_0(t;z))\big)=-DH(x_0(t;z))JD^2H(x_0(t;z)).
\end{equation}
Exploiting Eq. (\ref{eq:solFvarCL}), one obtains
\begin{equation}
\label{eq:a21redef}
\begin{array}{r}
\langle DH(z),\Psi JDH(z) \rangle = \displaystyle \int_0^{m\tau} \Big\langle DH(x_0(t;z)),\Big( JD^3H(x_0(t;z),x_1(t;z,\tilde{z})) +\\ + \partial_x g(x_0(t;z),t;m\tau/l,0)\Big) JDH(x_0(t;z) \Big\rangle dt
\end{array}
\end{equation}
and by taking the time derivative of the ODE in Eq. (\ref{eq:FVare}), one has the identity
\begin{equation}
\begin{array}{r}
\Big( JD^3H(x_0(t;z),x_1(t;z,\tilde{z}))+\partial_x g(x_0(t;z),t;m\tau/l,0)\Big) JDH(x_0(t;z)) = \\ \ddot{x}_1(t;z,\tilde{z})-JD^2H(x_0(t;z))\dot{x}_1(t;z,\tilde{z})-\partial_t g(x_0(t;z),t;m\tau/l,0).
\end{array}
\end{equation}
Substituting the latter result in Eq. (\ref{eq:a21redef}), one can use integration by parts to obtain
\begin{equation}
\begin{split}
\langle DH(z),\Psi JDH(z) \rangle= & \displaystyle -\int_0^{m\tau} \big\langle D_t\big(DH(x_0(t;z))\big),\dot{x}_1(t;z,\tilde{z})\big\rangle dt + \\\\
& \displaystyle-\int_0^{m\tau} \big\langle DH(x_0(t;z)),JD^2H(x_0(t;z_0))\dot{x}_1(t;z,\tilde{z})\big\rangle dt+ \\\\
& \displaystyle -\int_0^{m\tau} \big\langle DH(x_0(t;z)),\partial_t g(x_0(t;z),t;m\tau/l,0)\big\rangle dt,
\end{split}
\end{equation}
where we also used the fact that $\langle DH(x_0(t;z)), \dot{x}_1(t;z,\tilde{z})\rangle$ is $m\tau$-periodic. Moreover, the first two integrals cancel out due to Eq. (\ref{eq:DHFVadj}) and one finds
\begin{equation}
a_{21}=\langle DH(z),\Psi JDH(z) \rangle = -\int_0^{m\tau} \big\langle DH(x_0(t;z)),\partial_t g(x_0(t;z),t;m\tau/l,0)\big\rangle dt = M_{m:l}'(0),
\end{equation}
according to the definition of Remark A.4 in \cite{Cenedese2019}.
\end{proof}
\begin{proof}[Proof of Theorem \ref{thm:astability}]
We prove asymptotic stability by showing that under the conditions of the theorem all the Floquet multipliers lay within the unit circle in the complex plane for $\varepsilon > 0$ sufficiently small. 

Condition \textit{(i)} may be derived directly from the proof of Theorem \ref{thm:instability}. Indeed, in this case the two eigenvalues equal to $+1$ evolve as a complex pair with modulus less than $1$, cf Eq. (\ref{eq:modFMT}).

By assumption, the complex conjugated complex pairs related to the normal space are distinct for $\varepsilon=0$, so we can iteratively apply Theorem \ref{thm:eigSplit} considering $\mathcal{V}$ as each of the 2-dimensional strongly invariant subspaces $\mathcal{N}_k$ for $k=1,...,\,n-1$. Moreover, these eigenvalue pairs persist as complex ones since $P(\varepsilon)$ is real, so it is sufficient to evaluate their squared modulus to evaluate whether they move inside or outside the unit circle of the complex plane. As done in the proof of Theorem \ref{thm:instability}, this can be estimated as $1-\varepsilon m\tau C_{\mathcal{N}_k}$ at first order. Hence, all the eigenvalues of the normal space have modulus lower than $1$ if all the $n-1$ volume contractions related to the subspaces $\mathcal{N}_k$ are positive.
\end{proof}
We remark that, for the estimates of the last theorem to be valid, the value of $\varepsilon$ must be much smaller than the minimal separation between the pairs of eigenvalues of $\Pi$. In particular, in the vicinity of bifurcation or crossing points, these approximations may turn out to be weak.
\begin{remark}
If the unperturbed system is not in Hamiltonian form, then one can still prove that the quantity $T'(h) M_{m:l}'(0)$ governs the stability in the tangential directions. However, the formulas for the volume contractions are more complicated in this case. The Hamiltonian form provides drastic simplifications so that these contractions only depend on the pullback of the linear vector field $\partial_x g$ under $X_0$. Moreover, all the results we have proved in these sections, also apply to more general perturbations of Hamiltonian systems rather than the specific form we assumed in section \ref{sec:S2}.
\end{remark}

\subsection{Proof of Proposition \ref{prop:uniformcontraction}}
\label{app:A1S4}
\begin{proof}
With the shorthand notation
\begin{equation}
\begin{array}{c}
\partial_q Q = \partial_q Q(q_0(t;z),F(q_0(t;z),p_0(t;z)),t;m\tau/l,0), \\ \partial_p Q = \partial_p Q(q_0(t;z),F(q_0(t;z),p_0(t;z)),t;m\tau/l,0),
\end{array}
\end{equation}
we split
\begin{equation}
\begin{array}{c}
\partial_x g(x_0(t,z),t;m\tau/l,0) = \begin{bmatrix} 0  & 0 \\ \partial_q Q & \partial_p Q\end{bmatrix} =J A_q(t)+A_p(t), \\ \\
\begin{array}{lr}
A_q(t)=\begin{bmatrix} -\partial_q Q & 0 \\ 0 & 0\end{bmatrix}, & 
A_p(t)=\begin{bmatrix} 0 & 0 \\ 0 & \partial_p Q \end{bmatrix},
\end{array}
\end{array}
\end{equation}
so that, using Eq. (\ref{eq:traceeq}), we have
\begin{equation}
\begin{array}{ll}
C_\mathcal{V} = &\displaystyle -\frac{1}{2 m\tau v}\int_0^{m\tau} \mathrm{trace}\left(S_\mathcal{V}X^{-1}_0(t;z)\left(JA_q(t)-JA_q^\top (t)\right)X_0(t;z)R_\mathcal{V}\right)dt\\ \\ & \displaystyle -\frac{1}{2 m\tau v}\int_0^{m\tau} \mathrm{trace}\left(S_\mathcal{V}X^{-1}_0(t;z)\left(A_p(t)-JA_p^\top (t)J\right)X_0(t;z)R_\mathcal{V}\right)dt.
\end{array}
\end{equation}
For $A_q(t)=A_q^{\top}(t)$, the first of the integrals above vanishes. By using the fact that
\begin{equation}
A_p(t)-JA_p^\top (t)J=\begin{bmatrix} \partial_p Q^\top  & 0 \\ 0 & \partial_p Q \end{bmatrix},
\end{equation}
and by substituting $\partial_p Q = -\alpha I$, we obtain $C_\mathcal{V} =\alpha$ as claimed. This result agrees with the symmetry property of the Lyapunov spectra of conformal Hamiltonian systems \cite{Dressler1987}. 
\end{proof}

\end{appendix}

\bibliographystyle{unsrt}
\bibliography{Biblio}
\end{document}